\documentclass{amsart}
\usepackage{amsfonts}
\usepackage{amsmath,amscd}
\usepackage{amssymb}
\usepackage{amsthm}
\usepackage{newlfont}
\newcommand{\f}{\frac}

\newcommand{\ds}{\displaystyle}

\newtheorem{thm}{Theorem}[section]
\newtheorem{cor}[thm]{Corollary}
\newtheorem{lem}[thm]{Lemma}
\newtheorem{prop}[thm]{Proposition}
\theoremstyle{definition}

\theoremstyle{remark}

\numberwithin{equation}{section}

\begin{document}

\title[$2$-capability and $2$- nilpotent multiplier of finite dimensional nilpotent Lie algebras]{$2$-capability and $2$-nilpotent multiplier of finite dimensional nilpotent Lie algebras}
\author[P. Niroomand]{Peyman Niroomand}

\address{School of Mathematics and Computer Science\\
Damghan University, Damghan, Iran}
\email{niroomand@du.ac.ir, p$\_$niroomand@yahoo.com}

\author[M. Parvizi]{Mohsen Parvizi}
\address{Department of Pure Mathematics, Ferdowsi University of Mashhad, Mashhad, Iran.}
\email{parvizi@math.um.ac.ir}

%\subjclass{Primary ; Secondary}


\date{\today}

%\dedicatory{}

\begin{abstract} In the present context, we investigate to obtain some more results about $2$-nilpotent multiplier $\mathcal{M}^{(2)}(L)$ of a finite dimensional nilpotent Lie algebra $L$. For instance, we characterize the structure of $\mathcal{M}^{(2)}(H)$ when $H$ is a Heisenberg Lie algebra. Moreover, we give some inequalities on $ \mathrm{dim}~ \mathcal{M}^{(2)}(L)$ to reduce a well known upper bound on $2$-nilpotent multiplier as much as possible. Finally, we show that $H(m)$ is 2-capable if and only if m=1.

\end{abstract}

%%% ----------------------------------------------------------------------
\maketitle
%%% ----------------------------------------------------------------------

\section{Introduction}

For a finite group $G$, let $G$ be the quotient of a free group $F$ by a normal
subgroup $R$, then the $c$-nilpotent multiplier $\mathcal{M}^{(c)}(G)$ is defined as
\[R\cap\gamma_{c+1}(F)/\gamma_{c+1}[R,F],\]
in which $\gamma_{c+1}[R,F]=[\gamma_{c}[R,F],F]$ for $c\geq 1$. It is an especial case of the Baer invariant \cite{ba} with respect to the variety of nilpotent
groups of class at most $c$. When $c=1$, the abelian group $\mathcal{M}(G)=\mathcal{M}^{(1)}(G)$ is more known as the Schur multiplier of $G$ and it is much more studied, for instance in \cite{ka, ni1, ni9}.

Since determining the $c$-nilpotent multiplier of groups can be used for the classification of group into
isoclinism classes$($see \cite{bay}$)$, there are multiple papers concerning this subject.

Recently, several authors investigated to develop some results on the group theory case to Lie algebra.
In \cite{sal}, analogues to the $c$-nilpotent multiplier of groups, for a given Lie algebra $L$, the $c$-nilpotent multiplier of $L$ is defined as
\[\mathcal{M}^{(c)}(L)=R\cap F^{c+1}/[R,F]^{c+1},\]
in which $L$ presented as the quotient of a free Lie algebra $F$ by an ideal $R$, $F^{c+1}=\gamma_{c+1} (F)$ and $[R,F]^{c+1}=\gamma_{c+1}[R,F]$.
Similarly, for the case $c=1$, the abelian Lie algebra $\mathcal{M}(L)=\mathcal{M}^{(1)}(L)$ is more studied by the first author and the others
$($see for instance \cite{es, es2, bos, el2, har1, har, ni7, ni5, ni6, yan}$)$.

The $c$-nilpotent multiplier of a finite dimensional nilpotent
Lie algebra $L$ is a new field of interest in literature.
The present context is involving the $2$-nilpotent multiplier of a finite dimensional nilpotent Lie algebra $L$.
The aim of the current paper is divided into several steps.
In \cite[Corollary 2.8]{sal}, by a parallel result to the group theory result, showed for every finite nilpotent Lie algebra $L$, we have
\begin{equation}\label{e1}
\mathrm{dim}(\mathcal{M}^{(2)}(L))+\mathrm{dim}(L^3)\leq\f{1}{3}n(n-1)(n+1).
\end{equation}
Here we prove that abelian Lie algebras just attain the bound \ref{e1}. It shows that always $\mathrm{Ker}~\theta=0$ in \cite[Corollary 2.8 (ii)a]{sal}.

Since Heisenberg algebras $H(m)$ $($a Lie algebra of dimension $2m+1$ with $L^2=Z(L)$ and
$\mathrm{dim}~(L^2) = 1)$ have interest in several areas of Lie algebra, similar to the result of
\cite[Example 3]{es2} and \cite[Theorem 24]{mo}, by a quite different way, we give explicit structure of $2$-nilpotent multiplier of these algebras.
Among the other results since the Lie algebra which attained the upper bound \ref{e1} completely described in Lemma \ref{ab} $($they are just abelian Lie algebras$)$, by obtaining some new inequalities on dimension $\mathcal{M}^{(2)}(L)$,
we reduce bound \ref{e1} for non abelian Lie algebras as much as possible.

Finally, among the class of Heisenberg algebras, we show that which of them is $2$-capable. It means which of them is isomorphic to
$H/Z_2(H)$ for a Lie algebra $H$. For more information about the capability of Lie algebras see \cite{ni9, sal1}. These generalized the recently results for the group theory case in \cite{ni10}.

\section{Further investigation on $2$-nilpotent multiplier of finite dimensional nilpotent Lie algebra }
The present section illustrates to obtain further results on $2$-nilpotent multiplier of finite dimensional nilpotent Lie algebra. At first we give basic definitions and known results for the seek of convenience the reader.

Let $F$ be a free Lie algebra on an arbitrary totaly ordered set $X$. Recall from \cite{sh},
the basic commutator on the set $X$, which is defined as follows, is a basis of $F$.

The elements of $X$ are basic commutators of length one and ordered relative to the total order previously chosen.
Suppose all the basic commutators $a_i$ of length less than $k \geq 1$
have been defined and ordered. Then the basic commutators of length $k$ to be defined as
all commutators of the form $[a_i , a_j ]$ such that the sum of lengths of $a_i$ and $a_j$ is $k$, $a_i > a_j$, and if
$a_i = [a_s, a_t ]$, then $a_j \geq a_t$.
Also the number of basic commutators on $X$ of length $n$, namely $l_d(n)$, is
\[\f{1}{n}\sum_{m|n}\mu(m)d^{\f{n}{m}},\] where $\mu$ is the M\"{o}bius function.

From \cite{el2}, let $F$ be a fixed field, $L, K$ be two Lie algebras and $[ \ , \ ]$ denote the Lie bracket. By an action of $L$
on $K$ we mean an $F$-bilinear map \[(l,k) \in L\times K \mapsto
~^lk \in K~\text{satisfying}~\]
\[^{[l,l']}k= ~^l(~^{l'}k)- ~^{l'}(~^lk)~\text{and}~^l[k,k']=[~^lk,k']+[k,~^lk'], ~\text{for all}~ c \in F, l, l' \in L, k, k' \in K.\]
When $L$ is a subalgebra of a Lie algebra $P$ and $K$ is an ideal in $P$, then $L$
acts on $K$ by Lie multiplications $~^lk=[l,k]$.
A crossed module is a Lie homomorphism
$\sigma: K\rightarrow L$ together with an action of $L$ on $K$ such that
\[\sigma(^lk) = [l,\sigma(k)]~\text{and}~ ^{\sigma(k)}k' = [k,k'] ~\text{for all}~k,k'\in K ~\text{and}~ l\in L.\]

Let $\sigma: L\rightarrow M $ and $\eta: K\rightarrow M$ be two crossed modules, $L$ and $K$ act on each other and on themselves by Lie. Then these actions are called compatible provided that
\[~^{~^kl}k'=~^{k'}(~^lk)~\text{and}~^{~^lk}l'=~^{l'}(~^kl).\]

The non-abelian tensor product $L\otimes K$ of $L$ and $K$ is
the Lie algebra generated by the symbols $l\otimes k$ with defining
relations
\[c(l \otimes k)=cl \otimes k = l \otimes ck,
(l+l')\otimes k = l \otimes k + l' \otimes k,\]
\[l \otimes (k+k') = l\otimes k + l \otimes k',
~^ll' \otimes k = l \otimes ~^{l'}k -l' \otimes ~^lk,~ l \otimes ~^kk'= ~^{k'}l \otimes k - ~^kl
\otimes k',\]
\[[l\otimes k, l' \otimes k']=- ~^kl \otimes ~^{l'}k',~\text{for all}~c \in F, l, l' \in L, k, k' \in K.\]

The non-abelian tensor square of $L$ is a special case of tensor product
$L\otimes K$ when $K=L$. Note that we denote the usual
abelian tensor product $L \otimes_\mathbb{Z} K$, when $L$ and $K$
are abelian and the actions are trivial.

Let $L\square K$ be the submodule of $L\otimes K$ generated by the elements $l\otimes k$ such that $\sigma(l)=\eta(k)$. The factor Lie algebra $L\wedge K\cong L\otimes K/L\square K$ is called the exterior product of $L$ and $K$, and the image of $l\otimes k$
is denoted by $l\wedge k$ for all $l\in L,k \in K$. Throughout the paper $\Gamma$ is denoted the universal quadratic functor $($see \cite{el2}$)$.

Recall from \cite{ni9}, the exterior centre of a Lie algebra $Z^{\wedge}(L)=\{l\in L~|~l\wedge l'=1_{L\wedge L},~\forall~l'\in L \}$ of $L$.
It is shown that in \cite{ni9} the exterior centre $L$ is a central ideal of $L$ which allows us to decide when
Lie algebra $L$ is capable, that is, whether $L\cong H/Z(H)$ for a Lie algebra $H$.

The following Lemma is a consequence of \cite[Lemma 3.1]{ni9}.
\begin{lem} \label{ca} Let $L$ be a finite dimensional Lie algebra, $L$ is capable if and only if $Z^{\wedge}(L)=0$.
\end{lem}
The next two lemmas are special cases of  \cite[Proposition 2.1 (i)]{sal} when $c=2$ and that is useful for proving the next theorem.
\begin{lem}\label{ln}Let $I$ be an ideal in a Lie algebra $L$. Then the following sequences are exact.
\begin{itemize}
\item[(i)]$\mathrm{Ker}(\mu_I^2)\rightarrow \mathcal{M}^{(2)}(L)\rightarrow\mathcal{M}^{(2)}(L/I)\rightarrow \f {I\cap L^3}{[[I,L],L]}\rightarrow 0.$
    \item[(ii)]$(I\wedge  L/L^3)\wedge L/L^3 \rightarrow \mathcal{M}^{(2)}(L)\rightarrow\mathcal{M}^{(2)}(L/I)\rightarrow  I\cap L^3\rightarrow 0,$ when $[[I,L],L]=0$.
\end{itemize}
\end{lem}
\begin{lem}\label{l3}Let $I$ be an ideal of $L$, and put $K=L/I$. Then
\begin{itemize}
\item[(i)] $\mathrm{dim}~\mathcal{M}^{(2)}(K)\leq \mathrm{dim}~\mathcal{M}^{(2)}(L)+ \ds\mathrm{dim}~\f{I\cap L^3}{[[I,L],L]}.$
\item[(ii)] Moreover, if $I$ is a $2$-central subalgebra. Then

$(a).$ $(I\wedge L)\wedge L\rightarrow \mathcal{M}^{(2)}(L)\rightarrow ~\mathcal{M}^{(2)}(K)\rightarrow \mathrm{dim}~I\cap L^3\rightarrow 0$.

$(b).$ $ \mathrm{dim}~\mathcal{M}^{(2)}(L)+\mathrm{dim}~I\cap L^3 \leq \mathrm{dim}~\mathcal{M}^{(2)}(K)+\mathrm{dim}~(I\otimes L/L^3)\otimes L/L^3$.
\end{itemize}\end{lem}
\begin{proof} $(i)$. Using Lemma \ref{ln} (i).

$(ii)(a)$. Since $[I,L]\subseteq Z(L)$, $\mathrm{Ker}~\mu^2_I=(I\wedge L)\wedge L$ and $[[I,L],L]=0$ by Lemma \ref{ln}. It follows the result.

$(ii)(b)$. Since there is a natural epimorphism $(I\otimes L/L^3)\otimes L/L^3\rightarrow (I\wedge L/L^3)\wedge L/L^3 $, the result deduces from
Lemma \ref{ln} (ii).
\end{proof}
The following theorem gives the explicit structure of the Schur multiplier of all Heisenberg algebra.
\begin{thm}\cite[Example 3]{es2} $\mathrm{and}$ \cite[Theorem 24]{mo}\label{h}
Let $H(m)$ be Heisenberg algebra of dimension $2m+1$. Then
\begin{itemize}
\item[(i)]$\mathcal{M}(H(1))\cong A(2)$.
\item[(ii)]$\mathcal{M}(H(m))=A(2m^2-m-1)$ for all $m\geq 2$.
\end{itemize}
\end{thm}
The following result comes from \cite[Theorem 2.8]{ni8} and shows the behavior of 2-nilpotent multiplier respect to the direct sum of two Lie algebras.
\begin{thm}\label{ds}
Let $A$ and $B$ be finite dimensional Lie algebras. Then
\[\begin{array}{lcl}\mathcal{M}^{(2)}(A\oplus B)) &\cong& \mathcal{M}^{(2)}(A)\oplus ~\mathcal{M}^{(2)}(B)
\oplus \big((A/A^2\otimes_{\mathbb{Z}} A/A^2)\otimes_{\mathbb{Z}} B/B^2\big )\vspace{.3cm}\\&\oplus&\big((B/B^2\otimes_{\mathbb{Z}} B/B^2)\otimes_{\mathbb{Z}} A/A^2\big).\end{array}\]
\end{thm}
The following theorem is proved in \cite{sal} and will be used in the next contribution. At this point, we may give a short proof with a quite different way of \cite[Proposition 1.2]{sal} as follows.
\begin{thm}\label{ab}Let $L= A(n)$ be an abelian Lie algebra of dimension $L$. Then $\mathcal{M}^{(2)}(L)\cong A(\f{1}{3}n(n-1)(n+1))$.
\end{thm}
\begin{proof}
We perform induction on $n$. Assume $n=2$. Then Theorem \ref{ds} allows us to conclude that
\[\begin{array}{lcl}\mathcal{M}^{(2)}(L) &\cong& \mathcal{M}^{(2)}(A(1)) \oplus\mathcal{M}^{(2)}(A(1))
\oplus \big(A(1)\otimes_{\mathbb{Z}} A(1)\otimes _{\mathbb{Z}}A(1)\big )\vspace{.3cm}\\&\oplus&\big(A(1)\otimes_{\mathbb{Z}} A(1))\otimes _{\mathbb{Z}}A(1)\big)\cong A(1)\oplus A(1)\cong A(2).\end{array}\]

Now assume that $L\cong A(n)\cong A(n-1)\oplus A(1)$. By using induction hypothesis and Theorem \ref{ds}, we have
\[\begin{array}{lcl}\mathcal{M}^{(2)}(A(n-1)\oplus A(1))&\cong& \mathcal{M}^{(2)}(A(n-1))\oplus\big(A(n-1)\otimes_{\mathbb{Z}} A(n-1)\otimes _{\mathbb{Z}} A(1)\big)
\vspace{.3cm}\\&\oplus&\big(A(1)\otimes_{\mathbb{Z}} A(1)\otimes_{\mathbb{Z}} A(n-1)\big)\vspace{.3cm}\\&\cong&
A(\f{1}{3}n(n-1)(n-2))\oplus A((n-1)^2)\oplus A(n-1)\vspace{.3cm}\\&\cong& A(\f{1}{3}n(n-1)(n+1)).\end{array}\]
\end{proof}

The main strategy, in the next contribution, is to give the similar argument of Theorem \ref{h} for the $2$-nilpotent multiplier.
In the first theorem, we obtain the structure of $\mathcal{M}^{(2)}(L)$ when $L$ is non-capable Heisenberg algebra.
\begin{thm} Let $L=H(m)$ be a non-capable Heisenberg algebra. Then
\[\mathcal{M}^{(2)}(H(m))\cong A(\f{8m^3-2m}{3}).\]
\end{thm}
\begin{proof} Since $L$ is non-capable, Lemma \ref{ca} implies $Z^{\wedge}(L)=L^2=Z(L)$.
Invoking Lemma \ref{l3} by putting $I=Z^{\wedge}(L)$, we have $\mathcal{M}^{(2)}(H(m))\cong \mathcal{M}^{(2)}(H(m)/H(m)^2)$.
Now result follows from Theorem \ref{ab}.
\end{proof}
The following theorem from \cite[Theorem 3.4]{ni9} shows in the class of all Heisenberg algebras which one is capable.
\begin{thm}\label{ca1}$H(m)$ is capable if and only if $m = 1$.
\end{thm}
\begin{cor} \label{ca11}$H(m)$ is not $2$-capable for all $m\geq 2$.
\end{cor}
\begin{proof} Since every $2$-capable Lie algebra is capable, the result follows from Theorem \ref{ca1}.
\end{proof}
Since $H(m)$ for all $m\geq 2$ is not $2$-capable, we just need to discus about the $2$-capability of $H(1)$. Here, we obtain 2-nilpotent multiplier of $H(1)$ and in the next section we show $H(1)$ is $2$-capable.
\begin{thm} Let $L=H(1)$. Then
\[\mathcal{M}^{(2)}(H(1))\cong A(5).\]
\end{thm}
\begin{proof}
We know that $H(1)$ is in fact the free nilpotent Lie algebra of rank 2 and class 2. That is $H(1)\cong F/F^3$ in which $F$ is the free Lie algebra on 2 letters $x,y$. The second nilpotent multiplier of $H(1)$ is $F^4\cap F^3/[F^3,F,F]$ which is isomorphic to $F^3/F^5$ ant the latter is the abelian Lie algebra on the set of all basic commutators of weights 3 and 4 which is the set $\{[y,x,x],[y,x,y],[y,x,x,x],[y,x,x,y],[y,x,y,y]\}$. So the result holds.
\end{proof}
We summarize our result as below
\begin{thm}\label{th1}
Let $H(m)$ be Heisenberg algebra of dimension $2m+1$. Then
\begin{itemize}
\item[(i)]$\mathcal{M}^{(2)}(H(1))\cong A(5)$.
\item[(ii)]$\mathcal{M}^{(2)}(H(m))=A(\f{8m^3-2m}{3})$ for all $m\geq 2$.
\end{itemize}
\end{thm}The following Lemma lets us to obtain the structure of the $2$-nilpotent multiplier of all nilpotent Lie algebras with $\mathrm{dim}~
L^2=1$.
\begin{lem}\cite[Lemma 3.3]{ni7}\label{l1} Let $L$ be an $n$-dimensional Lie algebra and
$\mathrm{dim}~
L^2=1$. Then \[L\cong H(m)\oplus A(n-2m-1).\]
\end{lem}
\begin{thm}\label{mt1}Let $L$ be an $n$-dimensional Lie algebra with
$\mathrm{dim}~L^2=1$. Then \[\mathcal{M}^{(2)}(L) \cong \left\{\begin{array}{lcl} A(\frac{1}{3}n(n-1)(n-2)) & if\ m>1 ,\\
\\
A(\frac{1}{3}n(n-1)(n-2)+3) & if\ m=1.\end{array}\right.\]
\end{thm}
\begin{proof}By using Lemma \ref{l1}, we have $L\cong H(m)\oplus A(n-2m-1)$. Using the behavior of $2$-nilpotent multiplier respect to direct sum
\[\begin{array}{lcl}\mathcal{M}^{(2)}(L) &\cong& \mathcal{M}^{(2)}(H(m)) \oplus~\mathcal{M}^{(2)}(A(n-2m-1))
\vspace{.3cm}\\&\oplus&\big((H(m)/H(m)^2\otimes_{\mathbb{Z}} H(m)/H(m)^2)\otimes_{\mathbb{Z}} A(n-2m-1)\big )\vspace{.3cm}\\&\oplus&\big((A(n-2m-1)\otimes_{\mathbb{Z}} A(n-2m-1))\otimes _{\mathbb{Z}}H(m)/H(m)^2\big).\end{array}\]
First assume that $m=1$, then by virtue of Theorems \ref{ab} and \ref{th1}
\[\mathcal{M}^{(2)}(H(1))\cong A(5)~\text{and}~\mathcal{M}^{(2)}(A(n-3))\cong A(\f{1}{3}(n-2)(n-3)(n-4)).\]
Thus
\[\begin{array}{lcl}\mathcal{M}^{(2)}(L) &\cong& A(5) \oplus~A(\f{1}{3}(n-2)(n-3)(n-4))
\vspace{.3cm}\\&\oplus&\big(A(2)\otimes_{\mathbb{Z}} A(2))\otimes_{\mathbb{Z}} A(n-3)\big )\vspace{.3cm}\\&\oplus&\big((A(n-3)\otimes_{\mathbb{Z}} A(n-3))\otimes_{\mathbb{Z}} A(2)\big)
\vspace{.3cm}\\&\cong& A(\frac{1}{3}n(n-1)(n-2)+3).\end{array}\]The case $m\geq 1$ is obtained by a similar fashion.
\end{proof}

\begin{thm} Let $L$ be a $n$-dimensional nilpotent Lie algebra such that $\mathrm{dim}~L^2=m (m\geq 1)$. Then
\[\mathrm{dim}~\mathcal{M}^{(2)}(L)\leq \f13(n-m)\big((n+2m-2)(n-m-1)+3(m-1)\big)+3.\]
In particular, $\mathrm{dim}~\mathcal{M}^{(2)}(L) \leq \f13n(n-1)(n-2)+3$. The equality holds in last inequality if and only if $L\cong
H(1)\oplus A(n-3)$.
\end{thm}
\begin{proof} We do induction on $m$. For $m=1$, the result follows from Theorem \ref{mt1}. Let $m\geq 2$, and taking $I$ a $1$-dimensional central
ideal of $L$. Since $I$ and $L/L^3$ act to each other trivially we have $(I\otimes L/L^3)\otimes L/L^3\cong (I\otimes_{\mathbb{Z}} \f{L/L^3}{(L/L^3)^2})\otimes_{\mathbb{Z}} \f{L/L^3}{(L/L^3)^2})$. Thus by Lemma \ref{l3} $(ii)(b)$
\[\mathrm{dim}~\mathcal{M}^{(2)}(L)+\mathrm{dim}~I\cap L^3 \leq \mathrm{dim}~\mathcal{M}^{(2)}(L/I)+\mathrm{dim}~(I\otimes_{\mathbb{Z}} \f{L/L^3}{(L/L^3)^2})\otimes_{\mathbb{Z}} \f{L/L^3}{(L/L^3)^2}).\]
Since \[\mathrm{dim}~\mathcal{M}^{(2)}(L/I)\leq \f13(n-m)\big((n+2m-5)(n-m-1)+3(m-2)\big),\] we have
\[\begin{array}{lcl}\mathrm{dim}~\mathcal{M}^{(2)}(L)&\leq& \f13(n-m)\big((n+2m-5)(n-m-1)+3(m-2)\big)+3+(n-m)^2
\vspace{.3cm}\\&=&\f13(n-m)\big((n+2m-2)(n-m-1)+3(m-1)\big)+3,\end{array}\] as required.
\end{proof}
The following corollary shows that the converse of \cite[Proposition 1.2]{sal} for $c=2$ is also true. In fact it proves always
$\mathrm{Ker}~\theta=0$ in \cite[Corollary 2.8 (ii)a]{sal}.
\begin{cor}\label{lab}Let $L$ be a $n$-dimensional nilpotent Lie algebra. If $\mathrm{dim}~\mathcal{M}^{(2)}(L)=\f13n(n-1)(n+1)$, then $L\cong A(n)$.
\end{cor}
\section{2-capability of Lie algebras}
Following the terminology of \cite{ell2} for groups, a Lie algebra $L$ is said to be $2$-capable provided that
$L\cong H/Z_2(H)$ for a Lie algebra $H$. The concept $Z_2^{*}(L)$ was defined in \cite{salria} and it was proved that if $\pi:F/[R,F,F]\rightarrow F/R$ be a natural Lie epimorphism then
\[Z^{*}_2(L)=\pi(Z_2(F/[[R,F],F])),~ \text{for}~ c\geq 0 \]

The following proposition gives the close relation between $2$-capability and $Z^{*}_2(L)$.
\begin{prop}
A Lie algebra $L$ is $2$-capable if and only if $Z^{*}_2(L)=0.$ \end{prop}
\begin{proof}Let $L$ has a free presentation $F/R$, and $Z^{*}_2(L)=0$. Consider the natural epimorphism $\pi: F/[[R,F],F]\twoheadrightarrow F/R$.
Obviously\[\mathrm{Ker}~\pi=R/[[R,F],F]=Z_2(F/[[R,F],F]),\] and hence $L\cong F/[[R,F],F]/Z_2(F/[[R,F],F])$, which is a $2$-capable.

Conversely, let $L$ is $2$-capable and so $H/Z_2(H)\cong L$ for a Lie algebra $H$. Put $F/R\cong H$ and $Z_2(H)\cong S/R$. There is natural
epimorphism $\eta: F/[[S,F],F]\twoheadrightarrow F/S\cong L$. Since $Z_2(F/[[R,F],F])\subseteq \mathrm{Ker}~\eta$, $Z^{*}_2(L)=0$, as required.
\end{proof}
The following Theorem gives an instrumental tools to present the main.
\begin{thm}\label{ti}Let $I$ be an ideal subalgebra of $L$ such that $I\subseteq Z^{*}_2(L)$. Then the natural Lie homomorphism
$\mathcal{M}^{(2)}(L)\rightarrow \mathcal{M}^{(2)}(L/I)$ is a monomorphism.
\end{thm}
\begin{proof} Let $S/R$ and $F/R$ be two free presentations of $L$ and $I$, respectively. Looking the natural homomorphism
\[\phi:\mathcal{M}^{(2)}(L)\cong R\cap F^2/[[R,F],F]\rightarrow \mathcal{M}^{(2)}(L/I)\cong R\cap S^2/[[S,F],F]\] and the fact that
$S/R\subseteq Z_2(F/R)$ show $\phi$ has trivial kernel. The result follows.
\end{proof}
\begin{thm} A Heisenberg Lie algebra $H(m)$ is $2$-capable if and only if $m=1$.
\end{thm}
\begin{proof} Let  $m\geq 2$, by Corollary \ref{ca11} $H(m)$ is not capable so it is not $2$-capable as well. Hence we may assume that $L\cong H(1)$. Let $I$ be an ideal of $L$ od dimension 1. Then $L/I$ is abelian of dimension 2, and hence $\mathrm{dim}~\mathcal{M}^{(2)}(L)=2$. On the other hands, Theorem \ref{th1} implies $\mathrm{dim}~\mathcal{M}^{(2)}(L)=5$, and Theorem
\ref{ti} deduces $\mathcal{M}^{(2)}(L)\rightarrow \mathcal{M}^{(2)}(L/I)$ can not be a monomorphism, as required.
\end{proof}

\end{document}